\documentclass[sn-mathphys,Numbered]{sn-jnl}
\usepackage{graphicx}
\usepackage{multirow}
\usepackage{amsmath,amssymb,amsfonts}
\usepackage{amsthm}
\usepackage{mathrsfs}
\usepackage[title]{appendix}
\usepackage{xcolor}
\usepackage{textcomp}
\usepackage{manyfoot}
\usepackage{booktabs}
\usepackage{algorithm}
\usepackage{algorithmicx}
\usepackage{algpseudocode}
\usepackage{listings}
\usepackage{hyperref}
\usepackage{natbib}
\theoremstyle{thmstyleone}
\newtheorem{theorem}{Theorem}
\newtheorem{proposition}[theorem]{Proposition}
\theoremstyle{thmstyletwo}

\newtheorem{remark}{Remark}
\newtheorem{corollary}[theorem]{Corollary}
\theoremstyle{thmstylethree}

\newtheorem{conjecture}[theorem]{Conjecture}

\begin{document}

\title[Permutation Polynomials over $\mathbb{Z}_{n}$ and  T-quaternion Rings ]{Permutation Polynomials over $\mathbb{Z}_{n}$ and T-quaternion Rings }

\author*[1]{\fnm{A.} \sur{Telveenus}}\email{telveenusa@gmail.com}

\affil*[1]{\orgdiv{Former Lecturer, International Study Centre}, \orgname{Kingston University}, \orgaddress{\street{Kingston Hill Campus}, \city{London}, \postcode{KT2 7LB},  \country{UK}}  \orgdiv{and Former HoD Mathematics, Fatima Mata National College}, \orgname{University of Kerala}, \orgaddress{\street{Kollam}, \city{Kerala},  \country{India}} }

\abstract{T-quaternion rings are a special class of commutative unital subrings of the quaternion rings over commutative rings. Permutation Polynomials (PP) are typically defined over finite fields, such as $\mathbb{F}_{q}$,  where $q$ is a prime power. This article examines permutation polynomials over finite T-quaternion rings, focusing specifically on linear, quadratic, and cubic cases. The study begins with permutation polynomials over the rings $\mathbb{Z}_{n}$, before extending the discussion to T-quaternion rings.}



\maketitle
\section{Introduction}
William Rowan Hamilton was an Irish mathematician and physicist best known for his contributions to linear algebra, including the Cayley–Hamilton theorem. Although he is widely associated with the Cayley–Hamilton theorem, Hamilton is also remembered for inventing quaternions over the real numbers. This discovery became a landmark moment in the history of algebra. Hamilton’s insight into quaternions is famously commemorated on a bridge in Dublin, along a route he used to walk to his university. \\ Let $\mathbb{R}$ be the field of real numbers and $\mathbb{H}=\{ q=w+xi+yj+zk \ | w,x,y,z\in \mathbb{R}\}$. The addition is defined as component wise and the multiplication in $\mathbb{H}$ is defined by the following rules:\begin{enumerate}
                                 \item  $i^{2}=j^{2}=k^{2}=-1;$
                                 \item $ ij=k, \ jk=i, \ ki=j \ ;$
                                 \item $ ji=-k, \ kj=-i, \ ik=-j.  $ (\cite{h}, Book II, Chapter I,page 158)
                               \end{enumerate}
Clearly $\mathbb{H}$ is division ring over $\mathbb{R}$. $w$ and $xi+yj+zk$ are called the \textit{real} and \textit{imaginary} parts of $q$ or \textit{scalar quaternion} and \textit{vector quaternion} respectively. The rule 2 can be interpreted as $i \rightarrow j \rightarrow k \rightarrow i$ in a positive orientation(anticlockwise orientation). When the orientation is opposite, we get negative results as in rule 3. According to the above rules, the multiplication (Hamilton) of two elements $q_{1}=w_{1}+x_{1}i+y_{1}j+z_{1}k$ and $q_{2}=w_{2}+x_{2}i+y_{2}j+z_{2}k$ is given by 
$q_{1}q_{2}=(w_{1}w_{2}-x_{}x_{2}-y_{1}y_{2}-z_{1}z_{2})+(w_{1}x_{2}+w_{2}x_{1}+y_{1}z_{2}-z_{1}y_{2})i+(w_{1}y_{2}+w_{2}y_{1}-x_{1}z_{2}+z_{1}x_{2})j+(w_{1}z_{2}+z_{1}w_{2}+x_{1}y_{2}-x_{2}y_{1})k$
$q^{*}=w-xi-yj-zk$ is defined as the conjugate of $q=w+xi+yj+zk$ and the norm of $q$,$\mid\mid q \mid\mid$, is defined by the equation $\mid\mid q \mid\mid^{2}=qq^{*}=w^{2}+x^{2}+y^{2}+z^{2}$. If 
$\mid\mid q \mid\mid \neq 0$, then $q^{-1}=\dfrac{1}{w^{2}+x^{2}+y^{2}+z^{2}}(w-xi-yj-zk)$. If $q=w+xi+yj+zk=w+u$ with $\mid\mid u \mid\mid=1$, then $u^{2}=-1$ and thereby $u^{4}=1$. 
\section{Quaternions over commutative rings}
Let $\mathbf{R}$ be a commutative unital ring and let $U(\mathbf{R})$ be the unit group of $\mathbf{R}$. Let $NU(\mathbf{R})$ be the non-unit elements of $\mathbf{R}$. 
We define the ring of quaternions over $\mathbf{R}$ as $\mathcal{H_{\mathbf{R}}}=\{ q=w+xi+yj+zk \ | w,x,y,z\in \mathbf{R}\}$ and the rules of addition and multiplication remain the same as in the standard ring of quaternions. There are no changes in the definitions of conjugate and norm. \\
We state the following propostions and theorems without proof and for details refer to \cite{ath3} \\
\begin{proposition}
An element $q$ in $\mathcal{H_{\mathbf{R}}}$ is invertible in $\mathcal{H_{\mathbf{R}}}$ if and only if $\mid\mid q \mid\mid^{2} \in U(\mathbf{R})$. 
\end{proposition}
\begin{theorem}
Let $\mathcal{H_{\mathbf{R}}}=\{ q=w+xi+yj+zk \ | w,x,y,z\in \mathbf{R}\}$ be the ring of quaternions over $\mathbf{R}$. Then there exists an induced ring of quaternions $\mathcal{H_{\mathbf{R}}^{\perp}}$ with orientation just opposite to $\mathcal{H_{\mathbf{R}}}$ such that $\mathcal{H}_{\mathbf{R}}\cong\mathcal{H_{\mathbf{R}}^{\perp}}$. 
\end{theorem}
\begin{theorem}
If $q=w+xi+yj+zk=w+u \in \mathcal{H_{\mathbf{R}}}$ with $\mid\mid u \mid\mid=1$, then $q^{n}=q_{n}+q_{n}^{\prime}u$ where $\left( \begin{array}{c} q_{n} \\
q_{n}^{\prime}\end{array}\right)$=$\left(\begin{array}{cc} a & -1 \\ 1 & a\end{array} \right)^{n-1}$ $\left( \begin{array}{c} a \\
1 \end{array}\right)$ and $n$ is a positive integer.
\end{theorem}
\begin{theorem} \label{t1}
Let $\mathbf{R}$ be a finite commutative unital ring. Then $U(\mathcal{H_{\mathbf{R}}})=\{ q=w+xi+yj+zk \in \mathcal{H_{\mathbf{R}}} \ | q^{s}=1, \ \text{for some positive integer} \ s \}$ and $q^{s}$ can be computed using the following recurrence formula:\\
$$q^{s}=q_{s}+q_{s}^{\prime}i+q_{s}^{\prime \prime}j+q_{s}^{\prime \prime \prime }k $$ where \\
 
 $$ \left( \begin{array}{c}
 q_{s} \\
 q_{s}^{\prime} \\
 q_{s}^{\prime \prime} \\
 q_{s}^{\prime \prime \prime }
 \end{array}\right)
 = \left(\begin{array}{cccc} 
w & -x & -y& -z \\
x & w & z & -y \\
y & -z & w & x \\
z & y & -x & w 
\end{array} \right)
 \left( \begin{array}{c}
 q_{s-1} \\
 q_{s-1}^{\prime} \\
 q_{s-1}^{\prime \prime} \\
 q_{s-1}^{\prime \prime \prime }

  \end{array}\right)$$

 with  
$\left( \begin{array}{c}
 q_{1} \\
 q_{1}^{\prime} \\
 q_{1}^{\prime \prime} \\
 q_{1}^{\prime \prime \prime }

  \end{array}\right)$= $\left( \begin{array}{c}
 w \\
 x\\
 y \\
 z

  \end{array}\right).$
  
\end{theorem}
\begin{proposition}
Let $\mathbf{R}=\mathbb{Z}_{n}$, the ring of integers modulo $n$. If $n=2^{t}$, then $|U(\mathcal{H_{\mathbf{R}}})|=\phi(|\mathcal{H_{\mathbf{R}}}|)$.
\end{proposition}
\begin{proposition} \label{rd21}
Let $P=w_{1}+x_{1}i+y_{1}j+z_{1}k$ and $Q=w_{2}+x_{2}i+y_{2}j+z_{2}k$ be any two elements of $\mathcal{H}_{\mathbf{R}}$, where $\mathbf{R}$ is a commutative unital ring and if $\mathbf{R }$ is an integral domain or $2 \in U(\mathbf{R})$, then $PQ=QP$ $\iff$ the following simultaneous equations are satisfied:\\ $$y_{1}z_{2}=z_{1}y_{2}$$ $$z_{1}x_{2}=x_{1}z_{2}$$ $$x_{1}y_{2}=y_{1}x_{2}$$. 
\end{proposition}
\noindent
We define $\mathcal{S}_{\mathbf{R}}=\{ w+a(i+j+k) | w,a \in \mathbf{R}\}$. Clearly the elements of $\mathcal{S}_{\mathbf{R}}$ are satisfying the conditions laid in proposition \ref{rd21} and therefore it forms a commutative subring of $\mathcal{H}_{\mathbf{R}}$ and its group of units given by $U(\mathcal{S}_{\mathbf{R}})=\{ w+a(i+j+k) | w^{2}+3a^{2} \in U(\mathbf{R})\}$. We call $\mathcal{S}_{\mathbf{R}}$ as \textit{T-quaternion ring}. 
\begin{proposition}
The T-quaternion $\mathcal{S}_{\mathbf{R}}$ is a commutative unital subring of $\mathcal{H}_{\mathbf{R}}$ which is not an ideal of $\mathcal{H}_{\mathbf{R}}$.
\end{proposition}
\begin{theorem} \label{rd22}
The T-quaternion ring $\mathcal{S}_{\mathbf{Z}_{n}}$ is a field if and only if $n>3$, $n$ is a prime with $n-3$ has no square roots in $\mathbf{Z}_{n}$. Also, \\
 $ \begin{array}{ll}
    |U(\mathcal{S}_{\mathbf{Z}_{n}})| &= n^{2}-1 \quad \ \ \hbox{if $n$ is a prime and $(n-3)$ does not have  a square root in $\mathbf{Z}_{n}$;}  \\
          & =(n-1)^{2} \quad \hbox{if $n$ is a prime and $(n-3$) has  square roots in $\mathbf{Z}_{n}$.}  
   \end{array}$
\end{theorem}
\noindent
The conjecture given below is a reduction formula to compute $|U(\mathcal{S}_{\mathbf{Z}_{p^{e}}})|$. 
\begin{conjecture} \label{rd23}
Let $p$ be a prime and $e>0$ be an integer. Then $|U(\mathcal{S}_{\mathbf{Z}_{p^{e}}})|=p^{2(e-1)} \times |U(\mathcal{S}_{\mathbf{Z}_{p}})|$.
\end{conjecture}
\noindent
Assuming that the above conjecture is true. Then we will have the following theorem. 
\begin{theorem} \label{rd24}
If $n=p_{1}^{e_{1}}p_{2}^{e_{2}} \cdots p_{l}^{e_{l}}$, then $|U(\mathcal{S}_{\mathbf{Z}_{n}})|=\prod_{r=1}^{l} (p_{r}^{2(e_{r}-1)} \times |U(\mathcal{S}_{\mathbf{Z}_{p_{r}}})|)$.
\end{theorem}
\begin{corollary}
If $n=2^{s}3^{t}$ with integers $s,t \geq 0$, $|U(\mathcal{S}_{\mathbf{Z}_{n}})| =n\phi(n).$
\end{corollary}
\begin{corollary}
If $n=p_{1}^{e_{1}}p_{2}^{e_{2}} \cdots p_{l}^{e_{l}}$ with $3<p_{1}<p_{2}< \cdots <p_{l}$ and if all $p_{r}-3$ have square roots in $\mathbf{Z}_{p_{r}}$ for $r=1,2, \cdots l, $ then $|U(\mathcal{S}_{\mathbf{Z}_{n}})|=(\phi(n))^{2}.$
\end{corollary}
\section{Permutation Polynomials over $\mathbb{Z}n$}
For many values of $n$, the ring $\mathbb{Z}_{n}$ is not a field and it is a field only when $n$ is prime. In these cases, the zero divisors in $\mathbb{Z}_{n}$ play a crucial role in the construction of permutation polynomials.
A polynomial $f(X) \in \mathbb{Z}_{n}[X]$ is said to be a \textit{permutation polynomial} if $f:\mathbb{Z}_{n} \mapsto \mathbb{Z}_{n}$ is a bijection. If $f(X)+X$ is also a permutation polynomial, it is called a \textit{complete permutation polynomial}. 
\begin{theorem} \label{rd25}
If $n=p^{e}$ with prime $p>3$ and $e>1$, then there are $n \phi(n)$ linear, $[n-\phi(n)-1]n \phi(n)$ quadratic, and $[n-\phi(n)-1][n-\phi(n)]n \phi(n)$ cubic permutation polynomials in $\mathbb{Z}_{n}[X]$.
\end{theorem}
\begin{proof}
Case 1: Linear \\
Let $f(X)=AX+B$ be the linear permutation polynomial. The condition is $f(X_{1})=f(X_{2}) \quad \implies X_{1}=X_{2}$. Since the ring $\mathbb{Z}_{n}$ is finite, no need of checking the map $f$ is surjective. Thus we have $(X_{1}-X_{2})A=0$. So, $X_{1}=X_{2}$ is possible only when $A \in U(\mathbb{Z}_{n})$. This means that $A$ can take $\phi(n)$ and $B$ can take any values in $\mathbb{Z}_{n}$. Hence the result. \\
Case 2: Quadratic \\
Let $f(X)=AX^{2}+BX+C$ be the quadratic permutation polynomial. Proceding as in case 1, we get $X_{1}=X_{2}$ only when $A(X_{1}+X_{2})+B=2AX_{1}+B=u  \in U(\mathbb{Z}_{n})$ for any value of $X_{1}$ in 
$\mathbb{Z}_{n}$. Thus we have the following equations:
$$B=u$$
$$2A+B=u$$
$$4A+B=u$$
$$\cdots \cdots $$
$$2(p^{e}-1)A+B=u$$
From the above equations we get $B=u$ and $2A=4A=\cdots =2(p^{e}-1)A=0$. This means that $A$ is a zero divisor in $\mathbb{Z}_{n}$ with $A \neq 0$. Thus $A$ can take $[n-\phi(n)-1]$ values, $B$ can take $\phi(n)$ values being a unit, and $C$ can take any values in  $\mathbb{Z}_{n}$. Hence the result. \\
Case 3: Cubic \\
Let $f(X)=AX^{3}+BX^{2}+CX+D$ be the cubic permutation polynomial. Proceding as in case 1, we get $X_{1}=X_{2}$ only when $A(X_{1}^{2}+X_{1}X_{2}+X_{2}^{2})+B(X_{1}+X_{2})+C)=3AX_{1}^{2}+2BX_{1}+C=u  \in U(\mathbb{Z}_{n})$ for any value of $X_{1}$ in 
$\mathbb{Z}_{n}$. Thus we have the following equations:
$$C=u$$
$$3A+2B+C=u$$
$$12A+4B+C=u$$
$$\cdots \cdots $$
$$3(p^{e}-1)^{2}A+B(p^{e}-1)+C=u$$
The presence of $3$ and $2$ in the equations might affect the number of solutions if we take $n=2^{e}$ or $n=3^{e}$ as they are zero divisors in $n=2^{e}$ or $n=3^{e}$. This is the reason why we have excluded $p=2$ and $p=3$. \\
From the above equations we get $C=u$,  $3A=6A=\cdots =3(p^{e}-1)^{2}A=0$, and $2B=4B=6B \cdots = 2(p^{e}-1)B=0$. This means that $A$ is a zero divisor in $\mathbb{Z}_{n}$ with $A \neq 0$. Thus $A$ can take $[n-\phi(n)-1]$ values, $B$ is also a zero divisor and can take $[n-\phi(n)]$ values including zero,  $C$ being a unit it can take $\phi(n)$ values, and $D$ can take any values in  $\mathbb{Z}_{n}$. Hence the result. \\
\end{proof}
\begin{proposition}
If $p>3$ a prime, then the number of quadratic permutation polynomials is zero in $\mathbb{Z}_{p}[X]$. 
\end{proposition}
\begin{proof}
The propositon follows from the proof of case 2 of theorem \ref{rd25} as the only zero divisor is zero and A must take that value. This means that there are no quadratic permutation polynomials in $\mathbb{Z}_{p}[X]$ . 
\end{proof}
\begin{remark}
The number of linear permutation polynomials over $\mathbb{Z}_{n}$ for any $n>0$ is $n\phi(n)$. See the programme-1 in appendix to verify. 
\end{remark}
\begin{theorem}
If $n=p^{e}$ with prime $p>3$ and $e>1$, then there are $n (\phi(n)-p^{e-1})$ linear, $[n-\phi(n)-1]n (\phi(n)-p^{e-1})$ quadratic, and $[n-\phi(n)-1][n-\phi(n)]n (\phi(n)-p^{e-1})$ cubic complete permutation polynomials in $\mathbb{Z}_{n}[X]$.
\end{theorem}

\begin{proof}
In each case, coefficient of X must a unit for permutation polynomials and for the complete permutation polynomials they must be twin units (consecutive). There are $p^{e-1}(p-2)$ twin units for $n=p^{e}$(for each block of $p$ numbers, there are $(p-2)$ twin units and $p^{e-1}$ blocks) which is same as $(\phi(n)-p^{e-1})$. And hence the theorem.
\end{proof}

\section{Permutation Polynomials over T-quaternion rings}
In this section, we restrict ourselves to the ring $\mathbf{R}=\mathbb{Z}_{n}$ for the discussion on permutation polynomials over T-quaternion rings, unless otherwise stated. 
\begin{theorem} 
If $n=p_{1}^{e_{1}}p_{2}^{e_{2}} \cdots p_{l}^{e_{l}}$, then the number of linear permutation polynomials over T-quaternion ring $\mathcal{S}_{\mathbb{Z}_{n}}$ is  $n^{2}\times\prod_{r=1}^{l} (p_{r}^{2(e_{r}-1)} \times |U(\mathcal{S}_{\mathbf{Z}_{p_{r}}})| )$ where $|U(\mathcal{S}_{\mathbf{Z}_{p_{r}}})|$ can be computed by the formulas given in theorem \ref{rd22}.
\end{theorem}
\begin{proof}
Let $f(X)=AX+B$ be the linear permutation polynomial. The condition is $f(X_{1})=f(X_{2}) \quad \implies X_{1}=X_{2}$.  Thus we have $(X_{1}-X_{2})A=0$. So, $X_{1}=X_{2}$ is possible only when $A \in U(\mathcal{S}_{\mathbb{Z}_{n}})$. So, by theorem \ref{rd24}, the number of ways $A$ can take the values in $\mathcal{S}_{\mathbb{Z}_{n}}$ is $\prod_{r=1}^{l} (p_{r}^{2(e_{r}-1)} \times |U(\mathcal{S}_{\mathbf{Z}_{p_{r}}})| )$  and $B$  can take any values in $\mathcal{S}_{\mathbb{Z}_{n}}$ which is $n^{2}$ . Hence the result. 
\end{proof}
\begin{proposition}
If $n$ is even, there are no linear complete permutation polynomials over T-quaternion ring $\mathcal{S}_{\mathbb{Z}_{n}}$.
\end{proposition}
\begin{proof}
Suppose $f(X)=AX+B$ be a permutation polynomial with $A=w+a(i+j+k) \in  U(\mathcal{S}_{\mathbb{Z}_{n}})$. To prove $f(X)$ is a complete permutation polynomial, we need to prove $A+1=(w+1)+a(i+j+k) \in  U(\mathcal{S}_{\mathbb{Z}_{n}})$.This means that both $w^{2}+3a^{2}$ and $(w+1)^{2}+3a^{2}$ are units in $\mathbb{Z}_{n}$. Let  $w^{2}+3a^{2}=u_{1}, (w+1)^{2}+3a^{2}=u_{2} \in U(\mathbb{Z}_{n})$. This implies $w=2^{-1}(u_{2}-u_{1}-1)$. Since $n$ is even, $2$ is not invertible in $\mathbb{Z}_{n}$,  $w$  does not exist in $\mathbb{Z}_{n}$ . And hence the proof.
\end{proof}
\begin{proposition}
If $n=3^{t}$, then there are $3^{2t-1}$ linear complete permutation polynomials over T-quaternion ring $\mathcal{S}_{\mathbb{Z}_{n}}$.
\end{proposition}
\begin{proof}
Since $n=3^{t}$, any multiples of 3 are zero divisors in $\mathbb{Z}_{n}$. $\therefore \quad w\neq 3h$ for any integer $h\geq 0$. When $w=3h-1$, $w^{2}+3a^{2}=(3h-1)^{2}+3a^{2}=1$ after removing the multiples of 3 and which is a unit in $\mathbb{Z}_{n}$. However, $(w+1)^{2}+3a^{2}=(3k)^{2}+3a^{2}$ is not a unit in $\mathbb{Z}_{n}$. When $w=3h-2$, $w^{2}+3a^{2}=(3h-2)^{2}+3a^{2}=4$ which is a unit in $\mathbb{Z}_{n}$. $(w+1)^{2}+3a^{2}=(3h-1)^{2}+3a^{2}=1$ which is also a unit in $\mathbb{Z}_{n}$. So, $w=3h-2$ is a possible value for a linear complete permutation polynomial. This shows that there are $3^{t-1}$ values for $w$ and $3^{t}$ values for a in $\mathbb{Z}_{n}$. Hence the proof. 
\end{proof}
\noindent
The following conjecture is about an interesting proportionality result.
\begin{conjecture}
If $n=p^{e}$ for a prime $p>3$ and $e\geq0$,  then the number of  linear complete permutation polynomials over T-quaternion ring $\mathcal{S}_{\mathbb{Z}_{n}}$ is proportional to $p^{2e-2}$.
\end{conjecture}
\begin{remark}
The following table gives the constant of proportionality \\ $C$ for the first $10$ primes $p>3$ (See programme-2 in the appendix).
\begin{tabular}{|c|c|c|c|c|c|c|c|c|c|c|}
  \hline
  $p$ & $5$ & $7$& $11$ & $13$ & $17$ \\
  \hline
  $C$ & $23=5^{2}-2$ &$25=5^{2}$ & $119=11^{2}-2$ & $121=11^{2}$ & $287=17^{2}-2$ \\
  \hline
  $p$ & $19$ & $23$ & $29$ & $31$ & $37$ \\
  \hline
   $C$ &$289= 17^{2}$ & $527= 23^{2}-2$& $839=29^{2}-2$ & $841=29^{2}$ & $1225=35^{2}$\\
   \hline
  
\end{tabular}

\noindent
A pattern has been observed until $ p=23$ and after that the pattern breaks. 
\end{remark}
\noindent
From here onwards we shall denote any element  $w+a(i+j+k)$ in $\mathcal{S}_{\mathbb{Z}_{n}}$ as an ordered pair$(w,a)$. 
\begin{proposition} \label{rd26}
If $n=2$, there are $24$ quadratic permutation polynomials over T-quaternion ring $\mathcal{S}_{\mathbb{Z}_{n}}$ of which $8$ are complete permutation polynomials. 
\end{proposition}
\begin{proof}
We know that if $f(X)$ is a permutation polynomial, then $f(X) + C$ is also a permutation polynomial, where C is a constant. So, let us consider a quadratic polynomial without the constant term.
Let $f(X)=AX^{2}+BX$ where A and B are in $\mathcal{S}_{\mathbb{Z}_{n}}$. According to the new notation, let us take $X=(w,a),A=(w_{1}$, $a_{1})$, and $B=(w_{2},a_{2})$. 
Then $f(X)=(w_{1}(w^{2}-3a^{2})+w_{2}w, a_{1}((w^{2}-3a^{2})+wa_{2})$. Here $X$ takes the values $(0,0)$, $(0,1)$, $(1,0)$, $(1,1)$ and correspondingly $f(X)$ takes the values $(0,0$),$A$,$A+B$,$B$.To keep this as a bijection, $A\neq (0,0)$, $B \neq (0.0)$, when $A$ takes the value $(0,1)$, $B$ must take the values $(1,0)$ and $(1,1)$, when $A$ takes the value $(1,0)$, $B$ must take the values $(0,1)$ and $(1,1)$, and when $A$ takes the value $(1,1)$, $B$ must take the values $(1,0)$ and $(1,1)$. Thus we have 6 different quadratic permutation polynomials f(X) and after adding the constants, we will get $24$ quadratic permutation polynomials. The polynomials $f(X)=(1,0)X^{2}+(1,1)X$ and $f(X)=(1,0)X^{2}+(0,1)X$ are complete permutation polynomials. After adding the constants, we will get $8$ complete quadratic permutation polynomials. 
\end{proof}
\begin{proposition}
If $n>2$, there are no quadratic permutation polynomials over T-quaternion ring $\mathcal{S}_{\mathbb{Z}_{n}}$.
\end{proposition}
\begin{proof}
From the proof of the proposition \ref{rd26}, we have $f(X)=(w_{1}(w^{2}-3a^{2})+w_{2}w, a_{1}(w^{2}-3a^{2})+wa_{2})=(w_{1}(w^{2}+(n-3)a^{2})+w_{2}w, a_{1}(w^{2}+(n-3)a^{2})+wa_{2}).$  
If $n$ is even, when $X=(0,(\frac{n}{2}-1))$ and $X=(0,(\frac{n}{2}+1))$, $f(X)$ will have the same image $(n-3)(\frac{n}{2}-1)^{2}A=(n-3)(\frac{n}{2}+1)^{2}A$ as $(\frac{n}{2}-1)^{2}=(\frac{n}{2}+1)^{2}$ in $\mathbb{Z}_{n}$ as $n$ is even. So, $f(X)$ cannot be a bijection. If $n$ is odd, when $X=(0,(\frac{n-1}{2}))$ and $X=(0,(\frac{n+1}{2}))$, $f(X)$ will have the same image $(n-3)(\frac{n-1}{2})^{2}A=(n-3)(\frac{n+1}{2})^{2}A$ as $(\frac{n-1}{2})^{2}=(\frac{n+1}{2})^{2}$ in $\mathbb{Z}_{n}$ as $n$ is odd. So, $f(X)$ cannot be a bijection. Hence the proposition.
\end{proof}
\begin{proposition} \label{rd27}
If $n=2$, there are $16$ cubic permutation polynomials over T-quaternion ring $\mathcal{S}_{\mathbb{Z}_{n}}$ of which none are complete permutation polynomials. 
\end{proposition}
\begin{proof}
Let $f(X)=AX^{3}+BX^{2}+CX$ where A,B, and C are in $\mathcal{S}_{\mathbb{Z}_{n}}$. According to the new notation, let us take $X=(w,a),A=(w_{1}$, $a_{1})$, $B=(w_{2},a_{2})$, and $C=(w_{3},a_{3})$. Then $f(X)=(w_{1}w(w^{2}-9a^{2})-9aa_{1}(w^{2}-a^{2})+w_{2}(w^{2}-3a^{2})-6a_{2}wa+w_{3}w-3a_{3}a, \\ \quad 3w_{1}a(w^{2}-a^{2}+a_{1}w(w^{2}-9a^{2})+2w_{2}wa+a_{2}(w^{2}-3a^{2})+w_{3}a+a_{3}w)$. \\ Since $n=2$ , $f(X)$ can be simplified as $$f(X)=((w^{2}+a^{2})(w_{1}w+aa_{1}+w_{2})+w_{3}q+a_{3}w, \quad (w^{2}+a^{2})(w_{1}a+wa_{1}+a_{2})+w_{3}q+a_{3}w).$$ Thus we have following elements and images. \\
\begin{tabular}{|c|c|c|}
  \hline
  $X$ & $f(X)$ & Possible values 
                              for a bijection \\ \hline
  $(0,0)$ & $(0,0)$ &only $(0,0)$ \\ \hline
 $(0,1)$ & $(a_{1}+w_{2}+a_{3}, \quad w_{1}+a_{2}+w_{3})$ & $(0,1)$ or $(1,0$)\\ \hline
 $(1,0)$ & $(w_{1}+w_{2}+w_{3}, \quad a_{1}+a_{2}+a_{3})$ &  $(0,1)$ or $(1,0$)\\ \hline
  $(1,1)$ & $(w_{3}+a_{3}, \quad w_{3}+a_{3})$  & only $(1,1)$\\
  \hline 
\end{tabular}
\vspace{0.5cm}

\noindent
Case 1: Let $(a_{1}+w_{2}+a_{3}, \quad w_{1}+a_{2}+w_{3})=(0,1)$. Then $(w_{1}+w_{2}+w_{3}, \quad a_{1}+a_{2}+a_{3})=(1,0)$ and $w_{3}+a_{3}=1$. Since $A \neq (0,0)$, the possible values of $(w_{1},a_{1})$ are  $(0,1)$,  $(1,0)$, and $(1,1)$. The following tables shows different solutions of the unknowns.
\begin{tabular}{|c|c|c|c|}
  \hline
 $(w_{1},a_{1})$ & $(w_{1},a_{1})$ & $(w_{1},a_{1})$ & Whether all equations are satisfied or not \\ \hline
  $(0,1)$ & $(0,0)$,$(1,1)$ & $(0,1)$,$(1,0)$  & Both solutions fail \\
 & & respectively& $w_{1}+a_{2}+w_{3}=0\neq 1$ \\ \hline
  $(1,0)$ &$(0,0)$,$(1,1)$ & $(0,0)$,$(0,1)$  &  First fails $w_{3}+a_{3}=0\neq 1$ and the \\ 
  & &respectively & second fails $w_{1}+a_{2}+w_{3}=0\neq 1$\\ \hline
 $(1,1)$ & $(0,0)$,$(1,1)$ & $(0,1)$,$(1,0)$  & All equations are satisfied \\
 & & respectively &  \\ \\ \hline
\end{tabular}

\vspace{0.5cm}
\noindent
Case 2: Let $(a_{1}+w_{2}+a_{3}, \quad w_{1}+a_{2}+w_{3})=(1,0)$. Then $(w_{1}+w_{2}+w_{3}, \quad a_{1}+a_{2}+a_{3})=(0,1)$ and $w_{3}+a_{3}=1$. The following tables shows different solutions of the unknowns.
\begin{tabular}{|c|c|c|c|}
  \hline
 $(w_{1},a_{1})$ & $(w_{1},a_{1})$ & $(w_{1},a_{1})$ & Whether all equations are satisfied or not \\ \hline
  $(0,1)$ & $(0,0)$,$(1,1)$ & $(1,0)$,$(0,1)$  & Both solutions fail \\
 & & respectively& $w_{1}+a_{2}+w_{3}=1\neq 0$ \\ \hline
  $(1,0)$ &$(0,0)$,$(1,1)$ & $(1,1)$,$(0,1)$  &  First fails $w_{3}+a_{3}=0\neq 1$ and the \\ 
  & &respectively & second fails $w_{1}+a_{2}+w_{3}=1\neq 0$\\ \hline
 $(1,1)$ & $(0,0)$,$(1,1)$ & $(1,0)$,$(0,1)$  & All equations are satisfied \\
 & & respectively &  \\  \hline
\end{tabular}
\vspace{0.5cm}

\noindent
Thus we have the following permutation polynomials over T-quaternion ring $\mathcal{S}_{\mathbb{Z}_{n}}$ : \\
1. $f(X)=(1,1)X^{3}+(0,0)X^{2}+(0,1)X$  \\
2. $f(X)=(1,1)X^{3}+(1,1)X^{2}+(1,0)X$ \\
3. $f(X)=(1,1)X^{3}+(0,0)X^{2}+(1,0)X$  \\
4. $f(X)=(1,1)X^{3}+(1,1)X^{2}+(0,1)X$ \\
By adding the constants we get $16$ quadratic permutation polynomials over T-quaternion ring $\mathcal{S}_{\mathbb{Z}_{n}}$ of which clearly none are complete permutation polynomials. 
\end{proof}
\noindent
It has been observed that there are a large number of cubic permutation and complete permutation polynomials  over T-quaternion ring $\mathcal{S}_{\mathbb{Z}_{n}}$, where $n>2$. Evidently they are proportional to $n^{2}$. \\
\textbf{Open Question}: Find a formula or a relation between cubic permutation  and complete permutation polynomials  over T-quaternion ring $\mathcal{S}_{\mathbb{Z}_{n}}$ for $n>2$. 

\section{Conclusion}
Permutation polynomials over finite fields have a wide range of applications in Cryptography. Extending permutation polynomials from finite fields to rings is a challenging task, as rings introduce additional structural complexities, such as zero divisors, that must be carefully addressed. Once these challenges are overcome, permutation polynomials over rings may support cryptographic applications that are more secure, robust, and reliable.
\section{Appendix}

Programme-1 \\
The programme in C++ given below will find the permutation polynomials up to degree 3. For higher values of $n$ more RAM is required. 
\begin{verbatim}
#include<iostream>
#include<cmath>
using namespace std;
int main() {
	cout << "To find up to third order permutation polynomials" << endl;
	long long int A, B, C, D, X, Y[2000], n,c=0,c1=0,c2=0,c3=0,i,k, new1 = 0;
	cout << "Enter n =" << endl;
	cin >> n;
	for (A = 0; A < n; A++) {
		for (B = 0; B < n; B++) {
			for (C = 0; C < n; ++C) {
				for (D = 0; D < n; D++) {
					for (X = 0; X < n; X++) {

						Y[X] = (A * X * X * X + B * X * X + C * X + D) % n;
					}
					for (i = 0; i < n; i++)
					{
						for (k = i + 1; k < n; k++)
						{
							if (Y[i] == Y[k])
							{
								new1 = 1; break;
							}
						}	if (new1 == 1) { break; }
					}
					if (new1 == 1)
					{
						new1 = 0;
					}
					else {
						for (X = 0; X < n; X++)
						{
				//cout << "X=" << X << " " << "Y=" << Y[X]
				//<< " " << "A=" << A << " " << "B=" << B << " " << "C=" << C << " " << "D=" <<D<< endl;

						}
      cout << "The PP is " << A << "X^3 + " << B << "X^2 + " << C << "X + " << D << endl; c++;
						if (C != 0 && B == 0 && A == 0) { c1++; }
						if (B != 0 && A == 0) { c2++; }
						if (A != 0) { c3++; }
					}

				}
			}

		}
	} cout << "Total number of PPs of degree one = " << c1 << endl; 
	cout << "Total number of PPs of degree two = " << c2 << endl; 
	cout << "Total number of PPs of degree three = " << c3 << endl; 
	cout << "Total PP up to third degree= " << c << endl;
		return 0;
}
\end{verbatim}
Programme-2 \\
\begin{verbatim}
#include<iostream>
#include<cmath>
using namespace std;
int main() {
	std::cout << "To find twin units for complete permutation polnomials." << endl;
	long long int w, a, C, D, X, Y[2000], n, c = 0, d,e,c1 = 0, c2 = 0, c3 = 0,p1,p2;
	std:: cout << "Enter n =" << endl;
	cin >> n;
	for (w = 0; w < n; w++) {
		for (a = 0; a < n; a++) {
			c1 = (w * w + (3 * a * a)) % n; 
         //cout << "c1= " << c1<<"w= "<< w << "," << "a= " << a <<endl;
			for (d = 1; d < n; d++) {
				p1 = (d * c1) % n;
				if (p1 == 1) {
					c2 = ((w + 1) * (w + 1) + (3 * a * a)) % n; 
                //cout << "c2= " << c2 <<"w= " << w << "," << "a= " << a << endl; 
					{
						for (e = 1; e < n; e++) {
							p2 = (e * c2)%n; if (p2 == 1) { 
               std::cout << "w= " << w << "," << "a= " << a << endl; c1 = 0; c2 = 0; c3++; }
						}
					}
				}
			}
		}
	} std::cout << "Number of twin units= " << c3 << endl; return 0;}
\end{verbatim}

\end{document}